\documentclass[envcountsame]{llncs}

\usepackage{amssymb}
\usepackage{amsmath}
\usepackage{amsfonts}
\usepackage{latexsym}
\usepackage{srcltx}

\usepackage{enumerate}

\def\R{\mathbb{R}}
\def\C{\mathbb{C}}
\def\Z{\mathbb{Z}}
\def\Q{\mathbb{Q}}

\def\F{\mathbb{F}}

\newcommand{\mus}{\mu^\mathrm{sym}}
\newcommand{\Ms}{M^\mathrm{sym}}
\newcommand{\ms}{m^\mathrm{sym}}

\begin{document}

\title{Shimura modular curves and asymptotic symmetric tensor rank of multiplication in any finite field}

\author{St\'ephane Ballet\inst{1} \and Jean Chaumine\inst{2} \and Julia Pieltant\inst{3}}
\institute{Aix-Marseille Universit\'e, CNRS IML FRE 3529  \\ 
Case 930, 13288 Marseille Cedex 9, France\\
\email{stephane.ballet@univ-amu.fr}
\and Universit\'e de la Polyn\'esie Fran\c{c}aise, GAATI EA 3893\\
 B.P. 6570, 98702 Faa'a, Tahiti, France \\
\email{jean.chaumine@upf.pf}
\and INRIA Saclay,
LIX, \'Ecole Polytechnique,
91128 Palaiseau Cedex, France\\
\email{pieltant@lix.polytechnique.fr}}

%

\maketitle

\begin{abstract}
We obtain new asymptotical bounds for the symmetric tensor rank of multiplication in any finite extension of any finite field~$\F_q$. 
In this aim, we use the symmetric Chudnovsky-type generalized algorithm applied on a family of Shimura modular 
curves defined over $\F_{q^2}$ attaining the Drinfeld-Vl\u{a}du\c{t} bound and on the descent of this family over the definition field $\F_q$.
\end{abstract}

\section{Introduction}

\subsection{General context} 
The determination of the tensor rank of multiplication in finite fields is a problem which has been widely studied over the past decades both for its theoretical and practical importance. Besides it allows one to obtain multiplication algorithms with a low bilinear complexity, which determination is of crucial significance in cryptography, it has also  its own interest in algebraic complexity theory. The pioneer work of D.V. and G.V. Chudnovsky \cite{chch} resulted in  the design of a Karatsuba-like algorithm where the interpolation is done on points of algebraic curves with a sufficient number of rational points over the ground field. Following these footsteps, several improvements and generalizations of this algorithm leading to ever sharper bounds have been proposed since by various authors \cite{baro1,arna1,ceoz,rand3}, and have required to investigate and combine different techniques and objects from algebraic geometry such as evaluations on places of arbitrary degree, generalized evaluations, towers of algebraic function fields\ldots~Furthermore, a lot of connexions with other topics have been made : Shparlinski, Tsfasman and Vl\u{a}du\c{t} \cite{shtsvl} have first developed a correspondence between decompositions of the tensor of multiplication and a family of linear codes with good parameters that they called \textsl{(exact) supercodes}. These codes, renamed \textsl{multiplication friendly codes}, had recently be more extensively studied  and exploited by Cascudo, Cramer, Xing and Yang \cite{cacrxiya} to obtain good asymptotic results on the tensor rank. Moreover they combined their notion of multiplication friendly codes with two newly introduced primitives for function fields over finite fields \cite{cacrxing}, namely the torsion limit and systems of Riemann-Roch equations, to get news results not only  on asymptotic tensor rank but also on linear secret sharing systems and frameproof codes. This stresses that the  tensor rank determination problem has just as many mathematical interests as consequences and applications in various domains of computer science.

\subsection{Tensor rank of multiplication}

Let $q=p^s$ be a prime power, $\F_q$ be the finite field with $q$ elements and $\F_{q^n}$
be the degree $n$ extension of $\F_q$. The multiplication of two elements of $\F_{q^n}$ is 
an \mbox{$\F_q$-bilinear} application from $\F_{q^n} \times \F_{q^n}$ onto $\F_{q^n}$.
Then it can be considered as an \mbox{$\F_q$-linear} application from the tensor product 
${\F_{q^n} \otimes_{\F_q} \F_{q^n}}$
onto $\F_{q^n}$. Consequently it can be also considered as an element 
$T$ of ${(\F_{q^n} \otimes_{\F_q} \F_{q^n})^\star \otimes_{\F_q} \F_{q^n}}$,
namely an element of ${{\F_{q^n}}^\star \otimes_{\F_q} {\F_{q^n}}^\star \otimes_{\F_q} \F_{q^n}}$.
More precisely, when $T$ is written
\begin{equation}\label{tensor}
T=\sum_{i=1}^{r} x_i^\star\otimes y_i^\star\otimes c_i,
\end{equation}
where the $r$ elements $x_i^\star$ and the $r$ elements $y_i^\star$
are in the dual ${\F_{q^n}}^\star$ of $\F_{q^n}$ and the $r$ elements $c_i$ are in $\F_{q^n}$,
the following holds for any ${x,y \in \F_{q^n}}$:
$$
x\cdot y=\sum_{i=1}^r x_i^\star(x) y_i^\star(y) c_i.
$$
Unfortunately, the decomposition (\ref{tensor}) is not unique. 

\begin{definition}
The minimal number of summands in a decomposition of the tensor $T$ of the multiplication
is called the bilinear complexity of the multiplication and is denoted by
$\mu_{q}(n)$:
$$
\mu_{q}(n)= \min\left\{r \; \Big| \; T=\sum_{i=1}^{r} x_i^\star\otimes y_i^\star\otimes c_i\right\}.
$$
\end{definition}

However, the tensor $T$ admits also a symmetric decomposition:
\begin{equation}\label{symtensor}
T=\sum_{i=1}^{r} x_i^\star\otimes x_i^\star\otimes c_i.
\end{equation}
 
\begin{definition}
The minimal number of summands in a symmetric decomposition of the tensor $T$ of the multiplication
is called the symmetric bilinear complexity of the multiplication and is denoted by
$\mus_{q}(n)$:
$$
\mus_{q}(n)= \min\left\{r \; \Big| \; T=\sum_{i=1}^{r} x_i^\star\otimes x_i^\star\otimes c_i\right\}.
$$
\end{definition}

One easily gets that ${\mu_q(n) \leq \mus_{q}(n)}$. 
We know some cases where \linebreak[4]${\mu_{q}(n) = \mus_{q}(n)}$ but to the best of our knowledge, no example is known where 
we can prove that ${\mu_{q}(n) < \mus_{q}(n)}$. However, better 
upper bounds have been established in the asymmetric case and this may suggest that
in general the asymmetric bilinear complexity of the multiplication and
the symmetric one are distinct. In any case, at the moment, we must consider separately 
these two quantities. Remark that from an algorithmic point on view, as well as for
some specific applications, a symmetric bilinear algorithm can be more interesting than
an asymmetric one, unless if {\it a priori}, the constant factor in the bilinear complexity 
estimation is a little worse.
In this note we study the asymptotic behavior of the symmetric bilinear complexity
of the multiplication. More precisely we study the two following quantities:
\begin{equation}\label{M1}
\Ms_q= \limsup_{k \rightarrow \infty}\frac{\mus_q(k)}{k}, 
\end{equation}
\begin{equation}\label{m1}
\ms_q=\liminf_{k \rightarrow \infty}\frac{\mus_q(k)}{k}. 
\end{equation}

%
%
\subsection{Known results}

The  bilinear complexity $\mu_q(n)$ of the multiplication in the $n$-degree extension of a finite field $\F_q$ is known for certain values of $n$.  In particular, S. Winograd \cite{wino3} and H. de Groote \cite{groo} have shown that this complexity is ${\geq 2n-1}$, with equality holding if and only if ${n \leq \frac{1}{2}q+1}$. 
Using the principle of the D.V. and G.V. Chudnovsky algorithm \cite{chch} applied to elliptic curves, M.A. Shokrollahi has shown in \cite{shok} that the symmetric bilinear complexity of multiplication is equal to $2n$ for ${\frac{1}{2}q +1< n < \frac{1}{2}(q+1+{\epsilon (q) })}$ where $\epsilon$ is the function defined by:
$$
\epsilon (q) = \left \{
	\begin{array}{l}
 		 \mbox{greatest integer} \leq 2{\sqrt q} \mbox{ prime to $q$, if $q$ is not a perfect square} \\
  		2{\sqrt q}\mbox{, if $q$ is a perfect square.}
	\end{array} \right .
$$

Moreover, U. Baum and M.A. Shokrollahi have succeeded in \cite{bash} to construct effective optimal algorithms of type Chudnovsky in the elliptic case. 

Recently in \cite{ball1}, \cite{ball3}, \cite{baro1}, \cite{balbro}, \cite{balb}, \cite{bach} and \cite{ball5} the study made by M.A. Shokrollahi has been  generalized to algebraic function fields of genus~$g$. \\

Let us recall that the original algorithm of D.V. and G.V. Chudnovsky introduced in \cite{chch} is symmetric by definition and  leads to the following theorem:

\begin{theorem}
Let $q=p^r$ be a power of the prime $p$. The symmetric tensor rank $\mus_q(n)$ of multiplication in any finite field $\F_{q^n}$ is linear with respect to the extension degree; more precisely, there exists a constant $C_q$ such that:
$$
\mus_q(n) \leq C_q n.
$$
\end{theorem}

General forms for $C_q$ have been established since, depending on the cases where $q$ is a prime or a prime power, a square or not\ldots~In order to obtain these good estimates for the constant $C_q$, S. Ballet has given in \cite{ball1} some easy to verify conditions allowing the use of the D.V. and G.V. Chudnovsky  algorithm. Then S. Ballet and R. Rolland have generalized in \cite{baro1} the algorithm using places of degree one and~two. The best finalized version of this algorithm in this direction is a generalization  introduced by N. Arnaud in \cite{arna1} and developed later by M. Cenk and F. \"Ozbudak in \cite{ceoz}. This generalization uses several coefficients, instead of just the first one in the local expansion at each place on which we perform evaluations. Recently, Randriambolona introduced in \cite{rand3} a new generalization of the algorithm, which allows asymmetry in the construction.\\

From the results of \cite{ball1} and the generalized symmetric
algorithm, we obtain (cf. \cite{ball1}, \cite{baro1}):
\begin{theorem} \label{theoprinc}
Let $q$ be a prime power and let $n>1$ be an integer. 
Let $F/\F_q$ be an algebraic function field of genus $g$ 
and $N_k$ be the number of places of degree $k$ in $F/\F_q$.
If $F/\F _q$ is such that $2g+1 \leq q^{\frac{n-1}{2}}(q^{\frac{1}{2}}-1)$ then:
\begin{enumerate}[1)]
	\item if $N_1 > 2n+2g-2$, then $$ \mus_q(n) \leq 2n+g-1,$$
	\item if there exists a non-special divisor of degree $g-1$ 
and ${N_1+2N_2>2n+2g-2}$, then $$\mus_q(n)\leq 3n+3g,$$
	\item if $N_1+2N_2>2n+4g-2$, then $$\mus_q(n)\leq 3n+6g.$$
\end{enumerate}
\end{theorem}

\begin{theorem}\label{chudmq}
Let $q$ be a square $\geq 25$. Then
$$\ms_q\leq 2\left(1+\frac{1}{\sqrt{q}-3}\right).$$
\end{theorem}

Moreover, let us recall a very useful lemma due to  D.V. and G.V. Chudnovsky~\cite{chch} and 
Shparlinski, Tsfasman and Vl\u{a}du\c{t}~\cite[Lemma 1.2 and Corollary 1.3]{shtsvl}.

\begin{lemma}\label{lemasyMqmq}
For any prime power $q$ and for all positive integers $n$ and $m$, one has 
$$
\mu_{q}(m)\leq\mu_q(mn)\leq \mu_q(n)\cdot \mu_{q^n}(m),
$$
$$
m_{q}\leq m_{q^n} \cdot \mu_{q}(n)/n,
$$
$$
M_{q}\leq M_{q^n}\cdot\mu_{q}(n).
$$
\end{lemma}

Note that these inequalities are also true in the symmetric case.
Recall the following definitions that will be useful in the sequel.
Let $F/\F_q$ be a function field over the finite field $\F_q$ and $N_1(F)$ be the number
of places of degree one of $F/\F_q$.
Let us define:
$$
N_q(g)= \max \Big\{N_1(F)\, \big |\, F \mbox{ is a function field over }\F_q \mbox{ of genus }g  \Big\}
$$
and
$$
A(q)=\limsup_{g\rightarrow +\infty} \frac{N_q(g)}{g}.
$$
We know that (Drinfeld-Vl\u{a}du\c{t} bound):
$$A(q) \leq q^{\frac{1}{2}}-1,$$
the bound being reached if and only if $q$ is a square.

\section{New upper bounds for ${\ms_q}$ and ${\Ms_q}$}

In this section, we give upper bounds for the asymptotical quantities $\Ms_q$ and $\ms_q$ which are 
defined respectively by (\ref{M1}) and (\ref{m1}). 
As was noted in \cite[p. 694]{cacrxing} and more precisely in \cite[Section 5]{cacrxing2} (cf. also \cite{piel}), 
Theorems 3.1 and 3.9 in \cite{shtsvl} are not completely
correct.
We are going to repair that in the  following two propositions.

\begin{proposition}
 Let $q$ be a prime power such that $A(q)>2$. Then
$$\ms_q \leq 2\left(1+\frac{1}{A(q)-2}\right).$$
\end{proposition}

\begin{proof}\label{newboundmq}
 Let ${\{F_s/\F_q\}_s}$ be a sequence of algebraic function fields defined over $\F_q$.
Let us denote by $g_s$ the genus of $F_s/\F_q$ and by $N_1(s)$ the number of places of degree 1
of $F_s/\F_q$. Suppose that the sequence $\left(F_s/\F_q\right)_s$ was chosen such that:
\begin{enumerate}
 \item $\lim_{s \rightarrow +\infty}g_s=+\infty$,
 \item $\lim_{s \rightarrow +\infty}\frac{N_1(s)}{g_s}=A(q)$.
\end{enumerate}
Let $\epsilon$ be any real number such that
$0 < \epsilon < \frac{A(q)}{2} -1$.
Let us define the following integer
 $$n_s=\left\lfloor\frac{N_1(s)-2g_s(1+\epsilon)}{2}\right\rfloor.$$
Let us remark that
$$N_1(s)=g_s A(q) + o(g_s),$$
$$\mbox{so }N_1(s)-2(1+\epsilon)g_s=g_s\left(\strut A(q)-2(1+\epsilon)\right)+o(g_s).$$
Then the following holds:
\begin{enumerate}
 \item there exists an integer $s_0$ such that for any $s \geq s_0$ the integer $n_s$
is strictly positive,
 \item for any real number $c$ such that $0<c<A(q)-2(1+\epsilon)$ there exists
an integer $s_1$ such that for any integer $s\geq s_1$ the following holds:
$n_s \geq \frac{c}{2}g_s$, hence $n_s$ tends to $+\infty$,
 \item there exists an integer $s_2$ such that for any integer $s\geq s_2$
the following holds: 
$2g_s+1 \leq q^{\frac{n_s-1}{2}}\left(q^{\frac{1}{2}}-1\right)$
and consequently there exists a place of degree $n_s$ (cf. \cite[Corollary 5.2.10 (c) p. 207]{stic}),
 \item the following inequality holds:
$N_1(s)> 2n_s+2g_s-2$ and consequently, using Theorem 
\ref{theoprinc}  we conclude that ${\mus_q(n_s)  \leq  2n_s+g_s-1}$.                                                                                   
\end{enumerate}
Consequently, 
$$\frac{\mus_q(n_s)}{n_s} \leq 2+\frac{g_s-1}{n_s},$$
so
$$\ms_q \leq 2+ \lim_{s \rightarrow +\infty}\frac{2g_s-2}{N_1(s)-2(1+\epsilon)g_s-2}
\leq 2\left( 1+ \frac{1}{A(q)-2(1+\epsilon)}\right).$$
This inequality holding for any $\epsilon >0$ sufficiently small, we then obtain the result. \qed
\end{proof}

\begin{corollary}\label{coromq1}
Let $q=p^m$ be a prime power such that $q \geq 4$.
Then
$$\ms_{q^2}\leq 2\left(1+\frac{1}{q-3}\right).$$
\end{corollary}

Note that this corollary lightly improves Theorem \ref{chudmq}. Now in the case of arbitrary $q$, we obtain:

\begin{corollary}\label{coromq2}
For any $q=p^m>3$,
$$\ms_{q}\leq 3\left(1+\frac{1}{q-3}\right).$$
\end{corollary}
\begin{proof}
For any $q=p^m>3$, we have $q^2=p^{2m}\geq 16$ and thus Corollary \ref{coromq1} gives 
$\ms_{q^2}\leq 2\left(1+\frac{1}{q-3}\right)$. Then, by Lemma \ref{lemasyMqmq}, we have 
$$\ms_q\leq \ms_{q^2}\cdot\mus_q(2)/2$$
which gives the result since $\mus_q(2)=3$ for any $q$. \qed
\end{proof}

\vspace{.5em}

Now, we are going to show that for $\Ms_q$ 
the same upper bound as for $\ms_q$ can be proved though only in the case of $q$ being an even power of a prime. 
However, we are going to prove that in the case of $q$ being an odd power of a prime, the difference between 
the two bounds is very slight.

\begin{proposition}\label{newbound}
Let $q=p^m$ be a prime power such that $q \geq 4$. Then
$$\Ms_{q^2}\leq 2\left(1+\frac{1}{q-3}\right).$$
\end{proposition}

\begin{proof}
Let $q=p^m$ be a prime power such that $q\geq4$. Let us consider two cases.
First, we suppose that $q=p$. 
Moreover, firstly, let us consider the characteristic $p$ such that $p\neq 11$. 
Then it is known (\cite{tsvl} and \cite{shtsvl})  that the curve $X_k=X_0(11\ell_k)$, 
where $\ell_k$ is the $k$th prime number, has a genus $g_k=\ell_k$ and satisfies 
${N_1(X_k(\F_{q^2}))\geq (q-1)(g_k+1)}$ where $N_1(X_k(\F_{q^2}))$ denotes the number 
of  rational points over $\F_{q^2}$ of the curve $X_k$.
Let us consider a sufficiently large $n$.
There exist two consecutive prime numbers $\ell_k$ and $\ell_{k+1}$ such 
that ${(p-1)(\ell_{k+1}+1)> 2n+2\ell_{k+1}-2}$ and ${(p-1)(\ell_k+1)\leq 2n+2\ell_k-2}$. 
Let us consider  the algebraic function field 
$F_{k+1}/\F_{p^2}$ associated to the curve $X_{k+1}$ of genus $\ell_{k+1}$ defined over $\F_{p^2}$. 
Let $N_i(F_{k}/\F_{p^2})$ be the number 
of places of degree $i$ of $F_{k}/\F_{p^2}$. Then we get 
${N_1(F_{k+1}/\F_{p^2})\geq (p-1)(\ell_{k+1}+1)> 2n+2\ell_{k+1}-2}$. 
Moreover, it is known that $N_n(F_{k+1}/\F_{p^2})>0$ for any integer $n$ sufficiently large. 
We also know that $\ell_{k+1}-\ell_k\leq \ell_k^{0,525}$ for any integer $k\geq k_0$ where $k_0$ 
can be effectively determined by \cite{bahapi}. Then there exists a real number ${\epsilon>0}$ such that 
${\ell_{k+1}-\ell_k=\epsilon \ell_k\leq  \ell_k^{0,525}}$ namely ${\ell_{k+1}\leq(1+\epsilon)\ell_k}$. 
It is sufficient to choose $\epsilon$ such that $\epsilon \ell_k^{0,475}\leq 1$.
 Consequently, for any integer $n$ sufficiently large, this\linebreak[4]algebraic function field $F_{k+1}/\F_{p^2}$ satisfies Theorem \ref{theoprinc}, 
and so \linebreak[4]${\mus_{p^2}(n)\leq 2n+\ell_{k+1}-1\leq 2n+(1+\epsilon)\ell_k-1}$ with ${\ell_k\leq \frac{2n}{p-3}-\frac{p+1}{p-3}}$. 
Thus, \linebreak[4]as ${n\rightarrow +\infty}$ then ${\ell_k\rightarrow +\infty}$ and ${\epsilon \rightarrow 0}$, so
we obtain ${\Ms_{p^2}\leq 2\left(1+\frac{1}{p-3}\right)}$. Note that for $p=11$, Proposition~4.1.20 in \cite{tsvl} 
enables us to obtain ${g_k=\ell_k+O(1)}$.

Now, let us study the more difficult case where $q=p^m$ with $m>1$. 
We use the Shimura curves as in \cite{shtsvl}. 
Recall the construction of this good family. Let $L$ be a totally real abelian over 
$\Q$ number field of degree $m$ in which $p$ is inert, thus the residue class field 
${\mathcal O}_L/(p)$ of $p$, where ${\mathcal O}_L$ denotes the ring of integers of $L$, 
is isomorphic to the finite field $\F_{q}$. Let $\wp$ be a prime of $L$ which does 
not divide $p$ and let $B$ be a quaternion algebra for which 
$$
B\otimes_{\Q}\R=\mathrm{M}_2(\R) \otimes \mathbb{H} \otimes \cdots \otimes \mathbb{H}
$$ 
where $\mathbb{H}$ is the skew field of Hamilton quaternions. Let $B$ be 
also unramified at any finite place if $(m-1)$ is even; let $B$ be also unramified 
outside infinity and $\wp$ if $(m-1)$ is odd. Then, over $L$ one can define the 
Shimura curve by its complex points ${X_{\Gamma}(\C)=\Gamma\setminus \mathfrak{h}}†$, where 
$\mathfrak{h}$ is the Poincar\'e upper half-plane and $\Gamma$ is the group of units 
of a maximal order ${\mathcal O}$ of $B$ with totally positive norm modulo its center. 
Hence,  the considered Shimura curve admits an integral model over $L$ and 
it is well known that its reduction $X_{\Gamma,p}(\F_{p^{2m}})$ modulo $p$ 
is good and is defined over the residue class field ${\mathcal O}_L/(p)$ of $p$, 
which is isomorphic to $\F_q$ since $p$ is inert in $L$. 
Moreover, by \cite{ihar1}, the number ${N_1(X_{\Gamma,p}(\F_{q^2}))}$ of $\F_{q^2}$-points 
of $X_{\Gamma,p}$ is such that ${N_1(X_{\Gamma,p}(\F_{q^2}))\geq (q-1)(g+1)}$, where $g$ denotes the 
genus of $X_{\Gamma,p}(\F_{q^{2}})$. 
Let now $\ell$ be a prime which is greater than the 
maximum order of stabilizers $\Gamma_z$, where $z \in \mathfrak{h}$ is a fixed point 
of $\Gamma$ and let ${\wp \nmid \ell}$. Let ${\Gamma_0(\ell)_\ell}$ be the following subgroup of ${\mathrm{GL}_2(\Z_\ell)}$:
$$
\Gamma_0(\ell)_\ell=\left \lbrace
\left (
\begin{array}{ll}
 a & b \cr
 c & d
\end{array}
\right )
\in \mathrm{GL}_2(\Z_\ell)\, ; \, c \equiv 0 \pmod{\ell} \right \rbrace .
$$
Suppose that $\ell$ splits completely in $L$. Then there exists an embedding ${L \rightarrow \Q_\ell}$ where $\Q_\ell$ denotes 
the usual $\ell$-adic field, and since ${B\otimes_{\Q} \Q_\ell=\mathrm{M}_2(\Q_\ell)}$, we have a natural map: 
$$\phi_\ell: \Gamma \rightarrow \mathrm{GL}_2(\Z_\ell).$$
Let $\Gamma_\ell$ be the inverse image of $\Gamma_0(\ell)_\ell$ in $\Gamma$ under $\phi_\ell$. 
Then $\Gamma_\ell$ is a subgroup of $\Gamma$ of index $\ell$. We consider the Shimura curve $X_\ell$ with 
$$X_\ell(\C)=\Gamma_\ell\setminus\mathfrak{h}.$$ 
It admits an integral model over $L$ and so can be defined over $L$. Hence, its reduction $X_{\ell,p}$ modulo $p$ is good 
and it is defined over the residue class field ${\mathcal O}_L/(p)$ of $p$, which is isomorphic to $\F_q$ since $p$ is inert in $L$. 
Moreover the supersingular $\F_p$-points of  $X_{\Gamma,p}$ split completely in 
the natural projection $$\pi_\ell: X_{\ell,p} \rightarrow X_{\Gamma,p}.$$ Thus, 
the number of rational points of $X_{\ell,p}(\F_{q^2})$ verifies: 
$$N_1(X_{\ell,p}(\F_{q^2}))\geq \ell(q-1)(g+1).$$ Moreover, since $\ell$ is greater 
than the maximum order of a fixed point of $\Gamma$ on $\mathfrak{h}$, the 
projection $\pi_\ell$ is unramified and thus by Hurwitz formula, $$g_\ell=1+\ell(g-1)$$ 
where $g_\ell$ is the genus of $X_\ell$ (and also of $X_{\ell,p}$).


Note that since the field $L$ is abelian over $\Q$, there exists an integer $N$ such that the   field $L$ 
is contained in a cyclotomic extension $\Q(\zeta_N)$ where $\zeta_N$ denotes a primitive root of unity with 
minimal polynomial $\Phi_{N}$. Let us consider the reduction $\Phi_{N,\ell}$ of  $\Phi_{N}$ modulo the prime $\ell$. 
Then, the prime $\ell$ is totally split in the integer ring of $L$ if and only if   the polynomial $\Phi_{N,\ell}$ 
is totally split in ${\F_{\ell}=\Z/\ell\Z}$ i.e. if and only if $\F_{\ell}$ contains the $N$th roots of  unity which is 
equivalent to ${N\mid \ell-1}$. Hence, any prime $\ell$ such that ${\ell \equiv 1 \pmod{N}}$ is totally split in $\Q(\zeta_N)$ 
and then in $L$. Since $\ell$ runs over primes in an arithmetical progression, the ratio of two consecutive prime numbers 
${\ell \equiv 1 \pmod{N}}$ tends to one.

Then for any real number ${\epsilon >0}$, there exists an integer $k_0$ such that for any integer ${k\geq k_0}$, 
${\ell_{k+1}\leq (1+\epsilon)\ell_k}$ where $\ell_k$ and $\ell_{k+1}$ are two consecutive prime numbers congruent 
to one modulo $N$. Then there exists an integer $n_{\epsilon}$ such that  
for any integer ${n\geq n_{\epsilon}}$, the integer $k$ such that the two following inequalities hold
$$
\ell_{k+1}(q-1)(g+1)> 2n+2g_{\ell_{k+1}}-2
$$
and
$$
\ell_k(q-1)(g+1)\leq 2n+2g_{\ell_k}-2,
$$
satisfies ${k\geq k_0}$; where ${g_{\ell_i}=1+\ell_i(g-1)}$ for any integer $i$.  
Let us consider  the algebraic function field $F_{k}/\F_{q^2}$ defined over the finite field $\F_{q^2}$ 
associated to the Shimura curve $X_{\ell_{k}}$ of genus $g_{\ell_{k}}$. Let $N_i(F_{k}/\F_{q^2})$ be the number 
of places of degree~$i$ of $F_{k}/\F_{q^2}$.
Then ${N_1(F_{k+1}/\F_{q^2})\geq \ell_{k+1}(q-1)(g+1) > 2n+2g_{\ell_{k+1}}-2}$ where $g$ is the genus 
of the Shimura curve ${X_{\Gamma,p}(\F_{q^{2}})}$.
Moreover, it is known that there exists an integer $n_0$ such that for any integer ${n\geq n_0}$,  ${N_n(F_{k+1}/\F_{q^2})>0}$. 
Consequently, for any integer ${n\geq \max(n_{\epsilon},n_0)}$ this algebraic function field \linebreak[4] ${F_{k+1}/\F_{q^2}}$ 
satisfies Theorem \ref{theoprinc} and so 
$$
\mus_{q^2}(n)\leq 2n+g_{\ell_{k+1}}-1\leq 2n+\ell_{k+1}(g-1)\leq 2n+(1+\epsilon)\ell_k(g-1)
$$
with
${\ell_k< \frac{2n}{(q-1)(g+1)-2(g-1)}}$. 
Thus, for any real number ${\epsilon >0}$ and for any\linebreak[4]   ${n\geq \max(n_{\epsilon},n_0)}$, we obtain ${\mus_{q^2}(n)\leq 2n+\frac{2n(1+\epsilon)(g-1)}{(q-1)(g+1)-2(g-1)}}$ 
which gives \linebreak[4]  ${\Ms_{q^2}\leq2\left(1+\frac{1}{q-3}\right)}$. \qed
\end{proof}

\begin{proposition}\label{newbound2}
Let $q=p^m$ be a prime power with odd $m$ such that ${q \geq 5}$.
Then
$$\Ms_{q}\leq 3\left(1+\frac{2}{q-3}\right).$$
\end{proposition}

\begin{proof}
It is sufficient to consider the same families of curves than in Proposition~\ref{newbound}. 
These families of curves $\{X_k\}$ are defined over the residue class field of $p$ which  is isomorphic to $\F_q$.
Hence, we can consider the associated algebraic function fields ${F_k/\F_q}$ defined over $\F_q$. 
If ${q=p}$, we have  $N_1(F_{k+1}/\F_{p^2})=N_1(F_{k+1}/\F_{p})+2N_2(F_{k+1}/\F_{p})\geq (p-1)(\ell_{k+1}+1)> 2n+2\ell_{k+1}-2$
since ${F_{k+1}/\F_{p^2}=F_{k+1}/\F_{p}\otimes_{\F_p} \F_{p^2}}$.  
Then, for any real number ${\epsilon >0}$ and for any integer $n$ sufficiently large, we have
${\mus_{p}(n)\leq 3n+3g_{\ell_{k+1}}\leq 3n+3(1+\epsilon)\ell_k}$ by Theorem \ref{theoprinc} 
since ${N_n(F_{k+1}/\F_{q^2})>0}$. Then, by using the condition ${\ell_k\leq \frac{2n}{p-3}-\frac{p+1}{p-3}}$, 
we obtain ${\Ms_{p}\leq 3\left(1+\frac{2}{p-3}\right)}$.
If ${q=p^m}$ with odd $m$, we have 
$N_1(F_{k+1}/\F_{q^2})=N_1(F_{k+1}/\F_{q})+2N_2(F_{k+1}/\F_{q})\geq \ell_{k+1}(q-1)(g+1)> 2n+2g_{\ell_{k+1}}-2$
since ${F_{k+1}/\F_{q^2}=F_{k+1}/\F_{q}\otimes_{\F_q} \F_{q^2}}$.
Then, for any real number ${\epsilon >0}$ and for any integer $n$ sufficiently large as in Proof of Proposition~\ref{newbound}, we have
\linebreak[4]${\mus_{q}(n)\leq 3n+3g_{\ell_{k+1}}\leq 3n+3(1+\epsilon)\ell_k(g-1)}$ by Theorem~\ref{theoprinc} 
since \linebreak[4] ${N_n(F_{k+1}/\F_{q^2})>0}$. 
Then, by using the condition ${\ell_k< \frac{2n}{(q-1)(g+1)-2(g-1)}}$
we obtain ${\Ms_{q}\leq 3\left(1+\frac{2}{q-3}\right)}$. \qed
\end{proof}

\begin{remark}
Note that in \cite[Lemma IV.4]{cacrxiya}, Elkies gives another  construction of a family $\{\chi_s\}^{\infty}_{s=1}$
of Shimura curves over $\F_q$ satisfying 
for any prime power $q$ and for any integer $t\geq 1$ the following conditions:
\begin{enumerate}[(i)]
   \item the genus $g(F_s)$ tends to $+\infty$ as $s$ tends to $+\infty$, where $F_s$ stands for the function field $\F_q(\chi_s)$,
   \item $\lim_{s \rightarrow +\infty}g(F_s)/g(F_{s-1}) = 1$,
   \item $\lim_{s \rightarrow +\infty}B_{2t}(F_s)/g(F_s) = (q^t-1)/(2t)$, where $B_{2t}(F_s)$ stands for the number of places of degree $2t$ in $F_s$. 
\end{enumerate}

However,  this construction is not sufficiently explicit to enable Cascudo and al. \cite{cacrxing2} (and \cite{cacrxiya}) to derive the best bounds in all the cases 
(cf.  Section  \ref{comparison}). Indeed, let us recall the construction of Elkies.

Let $q=p^r$ be a prime power and put $f=rt$.
Let $K$ be a totally real number field such that $K/\Q$ is a Galois extension of degree $f$ and $p$ is totally inert in $K$.
Let $B$ be a quaternion algebra over $K$ such that the set $\mathsf{S}$ of non-archimedean primes of $K$ that are ramified in $B$ is Galois invariant.
Note that $B$ can be constructed by taking $\mathsf{S}$ to be either the empty set for odd $f$,
or the set of primes lying over $p$ for even $f$ (see \cite{shtsvl}).

Let ${\ell \neq p}$ be a rational prime outside $\mathsf{S}$ such that $\ell$ is totally inert in $K$
(note that in \cite{shtsvl}, $\ell$ is chosen such that it is completely splitting).
Consider the Shimura curve ${X_0^B(\ell):=\Gamma_0(\ell\mathcal{O}_K)\backslash\mathfrak{h}}$,
where $\mathfrak{h}$ is the upper half-plane and ${\Gamma_0(\ell\mathcal{O}_K)}$ is the subgroup of the unit group of the maximal order of $B$
mapping to upper triangle matrices modulo ${\ell\mathcal{O}_K}$.
Then ${X_0^B(\ell)}$ is defined over the rational field $\Q$ and has a good reduction modulo $p$.
Thus, the reduction of ${X_0^B(\ell)}$ is defined over $\F_p$, and therefore over $\F_q$ as well.
This curve has at least ${(p^f-1)g_\ell}$ supersingular points over ${\F_{p^{2f}}=\F_{q^{2t}}}$,
where $g_\ell$ is the genus of ${X_0^B(\ell)}$.
One knows that the ratio ${g_\ell/\ell^f}$ tends  to a fixed number $a$ when $\ell$ tends to ${+\infty}$.
Now let ${\{\ell_s\}_{s=1}^{+\infty}}$ be the set of consecutive primes such that $\ell_s$ are totally inert in $K$ and ${\ell_s \notin \mathsf{S}}$.
By Chebotarev's density theorem,
we have ${\ell_s/\ell_{s-1} \rightarrow 1}$ as $s$ tends to $+\infty$.
Hence, ${g_{\ell_s}/g_{\ell_{s-1}} \rightarrow 1}$ as $s$ tends to $+\infty$.

For the family of function fields $\{F_s/\F_q\}$ of the above Shimura curves,
the number $N_{2t}(F_s)$ of $\F_{q^{2t}}$-rational places of $F_s$ satisfies
$$
\lim_{g(F_s) \rightarrow +\infty}\frac{N_{2t}(F_s)}{g(F_s)}=p^f-1=q^t-1.
$$
Moreover, (i) and (ii) are satisfied as well.

By the identity ${N_{2t}(F_s)=\sum_{i|2t}iB_i(F_s)}$, we get
\begin{eqnarray*}
\liminf_{g(F_s) \rightarrow +\infty}\frac{1}{g(F_s)}\sum_{i=1}^{2t}\frac{iB_i(F_s)}{q^t-1} & \geq
		& \liminf_{g(F_s) \rightarrow +\infty}\frac{1}{g(F_s)}\sum_{i|2t}\frac{iB_i(F_s)}{q^t-1}\\
		& = & \liminf_{g(F_s)\rightarrow +\infty}\frac{N_{2t}(F_s)}{g(F_s)(q^t-1)}=1.
\end{eqnarray*}
Thus, the inequality $${\liminf_{g(F_s) \rightarrow +\infty} \frac{1}{g(F_s)}\sum_{i=1}^{2t}\frac{iB_i(F_s)}{q^{t}-1} \geq 1}$$
is satisfied and consequently (iii) is also satisfied by \cite[Lemma IV.3]{cacrxiya}.
\end{remark}

\section{Comparison with the current best asymptotical bounds}\label{comparison}

In this section, we recall the results obtained in \cite[Theorem IV.6 and IV.7]{cacrxiya} and 
\cite[Theorem 5.18]{cacrxing2}  which are known to give the best current estimates for $\Ms_q$, 
and compare these bounds to those established in  Propositions \ref{newbound} and~\ref{newbound2}.

\subsection{Comparison with the bounds in  \cite{cacrxiya}} 
In \cite{cacrxiya}, the authors establish the following results:
\begin{theorem}\label{MQ1}
For any prime power $q\geq2$, one has
	\begin{equation}\label{xingq}
   	\Ms_q \leq \mus_q(2t) \frac{q^t-1}{t(q^t-5)}
  	 \end{equation}
  for any $t \geq 1$ as long as $q^t-5 > 0$, and 
	\begin{equation}\label{xingq2}
	\Ms_{q^2} \leq \mus_{q^2}(t) \frac{2(q^{t}-1)}{t(q^{t}-5)}
	\end{equation}
   for any $t \geq 1$ as long as ${q^{t}-5 > 0}$.
\end{theorem}
Let us show that our results are better than those of this theorem except for some small values of $q$.
	\subsubsection{Bounds over $\F_q$.}
The estimates obtained in \cite[Corollary IV.8.]{cacrxiya} show that (\ref{xingq}) gives better bounds than Proposition \ref{newbound2} for any ${q\leq13}$. Indeed, one has:
$$
\begin{array}{|c|c|c|c|c|c|c|}
	\hline
	q & 5 & 7 & 8 & 9 & 11 & 13 \\
	 \hline
 	 \Ms_q \mbox{ \cite[Cor. IV.8]{cacrxiya}} & 4.8 & 3.82 & 3.74 & 3.68 & 3.62 & 3.59 \\
	\hline
	 \Ms_q \mbox{ [Prop. \ref{newbound2}]} & 6 & 4.5 & 4.2 & 4 & 3.75 & 3.6 \\
	\hline
	\end{array}
$$

However, as soon as ${q\geq15}$, our estimate is sharper than (\ref{xingq}). Indeed, for ${q\geq15}$, Proposition \ref{newbound2} gives:
\begin{equation*}\label{Mq3.5}
\Ms_q \leq 3.5.
\end{equation*}
On the other hand, since ${\mus_q(2t)\geq4t-1}$, the best estimate that can be obtained with Bound (\ref{xingq}) is:
\begin{equation}\label{xingqdec}
\Ms_{q} \leq (4t-1)\cdot  \frac{q^t-1}{t(q^t-5)} = \left(4-\frac{1}{t}\right) \cdot \left(1+\frac{4}{q^t-5}\right).  
\end{equation}
Thus  one must have ${4-\frac{1}{t}<3.5}$ to obtain a better estimate than 3.5, which requires ${t=1}$. In this case, (\ref{xingqdec}) becomes:
$$
\Ms_q\leq 3\left(1+\frac{4}{q-5}\right)
$$
which is less precise than the bound of Proposition \ref{newbound2} for any $q\geq15$.

\subsubsection{Bounds over $\F_{q^2}$.}

For ${q= 4}$, Proposition \ref{newbound} gives
${\Ms_{q^2} \leq 4}$, which is less sharp than Bound (\ref{xingq2}) applied with ${t=4}$, which leads to ${\Ms_{q^2} \leq 3.56}$.

However, for any ${q\geq 5}$, Proposition \ref{newbound} gives better bounds than (\ref{xingq2}). Indeed, since ${\mus_{q^2}(t) \geq 2t-1}$, the best estimate that can be obtained with (\ref{xingq2}) is:
\begin{equation}\label{xingq2dec}
\Ms_{q^2} \leq 2(2t-1) \cdot \frac{q^t-1}{t(q^t-5)} = \left(4-\frac{2}{t}\right) \cdot \left(1+ \frac{4}{q^t-5}\right).
\end{equation}

Since Proposition \ref{newbound} gives ${\Ms_{q^2} \leq 3}$ for any  ${q\geq5}$, it is necessary to have  ${4-\frac{2}{t}<3}$ to obtain a better bound with (\ref{xingq2dec}), which requires ${t=1}$. This is impossible for ${q=5}$ since Bound (\ref{xingq2}) is undefined in this case, and for ${q>5}$ and ${t=1}$, (\ref{xingq2dec}) becomes:
$$
\Ms_{q^2} \leq 2\left( 1+ \frac{4}{q-5}\right)
$$
which is less sharp than the bound obtained from Proposition \ref{newbound}.

\subsection{Comparison with the bounds in  \cite{cacrxing2}}

In \cite{cacrxing2} (which is an extended version of \cite{cacrxing}), 
the authors establish the following asymptotic bounds:
\begin{theorem}\label{MQ2}
For a prime power $q$, one has
$$\Ms_q \leq
\left\{
\begin{array}{ll}
\mus_q(2t)\frac{q^t-1}{t(q^t-2-\log_q2)}, & \mbox{if } 2 | q \\
\mus_q(2t)\frac{q^t-1}{t(q^t-2-2\log_q2)}, & \mbox{otherwise}
\end{array}
\right
.$$
for any ${t \geq 1}$ as long as ${q^t-2-\log_q2 > 0}$ for even $q$;
and \linebreak[4]${q^t-2-2\log_q2 > 0}$ for odd $q$.
\end{theorem}

This bound always beats the one of Proposition \ref{newbound2} for arbitrary $q$ (for example, by setting ${t=1}$ and ${\mus_q(2t)=4t-1}$). Nevertheless, if we focus on the case of $\Ms_{q^2}$, then the bound of Proposition \ref{newbound} is better as soon as ${q>5}$ since in this case, it gives:
$$
\Ms_{q^2} <3
$$
which can not be reached with the bound of Theorem \ref{xing2}, since the best that one can get is:
\begin{equation*}\label{xing2}
\Ms_{q} \leq
\left\{
\begin{array}{ll}
\left(4-\frac{1}{t}\right)\left(1+\frac{1+ \log_q 2}{q^t-2-\log_q2}\right), & \mbox{if } 2 | q \\
\left(4-\frac{1}{t}\right)\left(1+\frac{1+2\log_q 2}{q^t-2-2\log_q2}\right), & \mbox{otherwise}
\end{array}
\right.
\end{equation*}
which obviously  can not be $<3$.

\end{document}